\documentclass[a4paper, 12pt]{amsart}
\usepackage{amsmath, amssymb, amsthm, graphicx, stmaryrd}

\usepackage[top=40mm,bottom=40mm,left=30mm,right=30mm]{geometry}

\theoremstyle{definition}
\newtheorem{theorem}{Theorem}[section]
\newtheorem{definition}[theorem]{Definition}

\newtheorem{proposition}[theorem]{Proposition}
\newtheorem{corollary}[theorem]{Corollary}
\newtheorem{remark}[theorem]{Remark}

\makeatletter
\@addtoreset{equation}{section} 
\makeatother

\DeclareMathOperator{\Diff}{Diff}
\DeclareMathOperator{\grp}{grp}
\DeclareMathOperator{\sing}{sing}

\title{Group cocycles on
the volume-preserving diffeomorphism group}
\author{Shuhei Maruyama}
\address{Graduate School of Mathematics,
Nagoya University, Japan}
\email{m17037h@math.nagoya-u.ac.jp}

\begin{document}

\begin{abstract}
  We construct two kinds of group cocycles
  on the volume-preserving diffeomorphism group.
  We show that, for the volume-preserving diffeomorphism group
  of the sphere,
  one of the cocycles gives the Euler class of flat
  sphere bundles.
\end{abstract}

\maketitle

\section{Introduction}

Let $M$ be a connected manifold and $n$ denote the
minimum positive number that the homology
$H_n(M;\mathbb{Z})$ is non-zero.
Let $\Omega$ denote a closed $n$-form on $M$
such that the cohomology class $[\Omega] \in H^n(M;\mathbb{R})$
is a non-zero image of the map
$H^n(M;\mathbb{Z}) \to H^n(M;\mathbb{R})$.
Let $G$ denote the group of $\Omega$-preserving diffeomorphisms
of $M$.
The normalized volume form and the group of
volume-preserving diffeomorphisms is an example.
In this paper, we construct two kinds of group-cocycles on $G$;
group $(n+1)$-cocycles $c_k$ with coefficients
in the trivial $G$-module $\mathbb{Z}$
(Definition \ref{def:c_k}),
and group $n$-cocycles $b_k$ with coefficients
in the trivial $G$-module $S^1 = \mathbb{R}/\mathbb{Z}$
(Definition \ref{def:b_k}).
We show that, if the manifold is the $n$-sphere $S^n$ and the $n$-form
$\Omega$ is the normalized standard volume form on $S^n$,
the cohomology class $[c_k]$ is equal to the Euler class of
flat sphere bundle up to sign
(Theorem \ref{thm:cocycle_euler_class}).
By using this, we also show that the group cohomology classes
$[c_k]$ and $[b_k]$ are non-zero for the $n$-sphere $S^n$
(Theorem \ref{thm:non-trivial_example}).

Some group cocycles have been constructed
on groups of diffeomorphisms that preserve a fixed differential form,
such as the symplectomorphism group and
the volume-preserving diffeomorphism group.
On the symplectomorphism group of an exact symplectic manifold,
Ismagilov, Losik, and Michor constructed
in \cite{ismagilov_losik_michor06} a group two-cocycle
with coefficients in $\mathbb{R}$.
On the symplectomorphism group of an integral symplectic manifold,
a group three-cocycle
with coefficients in $\mathbb{Z}$ is constructed
in \cite{maruyama19_DD}, which is
a variant of Ismagilov, Losik, and Michor's one.
Our group cocycles $c_k$ are considered as
generalizations of the cocycle in \cite{maruyama19_DD}.
On the group of diffeomorphisms that preserve a fixed exact form,
Losik and Michor constructed in \cite{losik_michor06}
a group cocycle with coefficients in $\mathbb{R}$.
Our group cocycles $b_k$ are analogous to Losik and Michor's one.

\section{Group cocycles}\label{sec:group_cocycle}
\subsection{Group cohomology}\label{subsec:cohomology}
Let $G$ be a group and $A$ be a $G$-module.
The set of all maps
$C_{\grp}^p(G;A) = \{ c : G^p \to A : \text{map} \}$
is called {\it the group $p$-cochains
of $G$ with coefficients in $A$}.
The coboundary operator
$\delta : C_{\grp}^p(G;A) \to C_{\grp}^{p+1}(G;A)$
is defined by
\begin{align*}
  \delta c (g_1, \dots, g_{p+1}) =&
  g_1 c(g_2, \dots, g_{p+1}) +
  \sum_{i = 1}^{p}(-1)^i c(g_1, \dots, g_i g_{i+1},
  \dots, g_{p+1}) \\
  & +(-1)^{p+1}c(g_1, \dots, g_p)
\end{align*}
if $A$ is a left $G$-module and
\begin{align*}
  \delta c (g_1, \dots, g_{p+1}) =&
  c(g_2, \dots, g_{p+1}) +
  \sum_{i = 1}^{p}(-1)^i c(g_1, \dots, g_i g_{i+1},
  \dots, g_{p+1}) \\
  & +(-1)^{p+1}c(g_1, \dots, g_p)g_{p+1}
\end{align*}
if $A$ is a right $G$-module.
The {\it group cohomology $H_{\grp}^*(G;A)$ of $G$
with coefficients in $A$} is the cohomology of the
cochain complex $(C_{\grp}^*,\delta)$.

Let $G^{\delta}$ denote the group $G$ with discrete topology.
Then the group cohomology of $G$ is isomorphic to the
singular cohomology of the classifying space $BG^{\delta}$
of $G^{\delta}$
(see \cite{brown82}).
Under this identification, a group cohomology class gives
a universal characteristic class of flat $G$-bundles.

\subsection{The group cocycles with coefficients in $\mathbb{R}$}\label{subsec:cocycle}
Let $M$ be a connected manifold and $n$ denote the
minimum positive number that the homology
$H_n(M;\mathbb{Z})$ is non-zero.
Let $\Omega$ denote a closed $n$-form on $M$
such that the cohomology class $[\Omega] \in H^n(M;\mathbb{R})$
is a non-zero image of the map
$H^n(M;\mathbb{Z}) \to H^n(M;\mathbb{R})$.
The typical example of $(M,\Omega)$ is a homology $n$-sphere and its
(normalized) volume form.
Let $G = \Diff_{\Omega}(M)$ denote the $\Omega$-preserving
diffeomorphism group.
We regard the integers $\mathbb{Z}$ as the trivial $G$-module.
Then we define group $(n+1)$-cocycles $c_k$ in
$C_{\grp}^{n+1}(G;\mathbb{Z})$ as follows.

Let $(C_*(M;\mathbb{Z}),\partial)$ and $(C^*(M;\mathbb{Z}),d)$ denote
the singular chain complex and the singular cochain complex
respectively.
We regard $C_q(M;\mathbb{Z})$ and $C^q(M;\mathbb{Z})$
as the left $G$-module and the right $G$-module respectively.
Let us consider the double complexes
$C_{\grp}^p(G;C_q(M;\mathbb{Z}))$ and
$C_{\grp}^p(G;C^q(M;\mathbb{Z}))$.
Take a point $\Delta_0 = x \in M = C_0(M;\mathbb{Z})$ and a
singular $n$-cocycle
$w_n \in C^n(M;\mathbb{Z}) = C_{\grp}^0(G;C^n(M;\mathbb{Z}))$
that the cohomology class $[w_n] \in H^n(M;\mathbb{Z})$
corresponds to the class $[\Omega]$ in $H^n(M;\mathbb{R})$.
By the assumption of $M$, we take elements
$\Delta_k \in C_{\grp}^k(G;C_k(M;\mathbb{Z}))$
for $0 \leq k < n$
satisfying
\[
  \delta \Delta_k = \partial \Delta_{k+1}
  \in C_{\grp}^{k+1}(G;C_k(M;\mathbb{Z})).
\]
Since the map $H^n(M;\mathbb{Z}) \to H^n(M;\mathbb{R})$
is injective, any element in $G$ preserves the cohomology class
$[w_n]$, that is, $\delta w_n (g) = w_n - g^* w_n$ is a coboundary
for any $g \in G$.
Thus we take an element $w_{n-1} \in C_{\grp}^1(G;C^{n-1}(M;\mathbb{Z}))$
such that
\[
  \delta w_n = - (-1) d w_{n-1}
  \in C_{\grp}^{1}(G;C^n(M;\mathbb{Z})).
\]
Since $H^k(M;\mathbb{Z}) = 0$ for $0 < k < n$ by the universal coefficients
theorem,
we take elements $w_k \in C_{\grp}^{n-k}(G;C^k(M;\mathbb{Z}))$
such that
\[
  \delta w_k = - (-1)^{n-k+1} d w_{k-1}
  \in C_{\grp}^{n-k+1}(G;C^k(M;\mathbb{Z})).
\]

Let $\langle \cdot ,\cdot \rangle : C^r(M;\mathbb{Z}) \times
C_r(M;\mathbb{Z}) \to \mathbb{Z}$ denote the pairing.
This map induces the map
\[
  \langle \cdot ,\cdot \rangle :
  C_{\grp}^p(G;C^r(M;\mathbb{Z})) \times
  C_{\grp}^q(G;C_r(M;\mathbb{Z})) \to
  C_{\grp}^{p+q}(G;\mathbb{Z}).
\]

\begin{definition}\label{def:c_k}
  For $0 \leq k \leq n$, define $c_k \in C_{\grp}^{n+1}(G;\mathbb{Z})$
  by $c_k = \langle \delta w_k, \Delta_k \rangle$.
\end{definition}

To show that the cochains $c_k$ are cocycles, we use the following
proposition.
\begin{proposition}\label{prop:chain_rule}
  For any $(a,b) \in C_{\grp}^p(G;C^r(M;\mathbb{Z})) \times
  C_{\grp}^q(G;C_r(M;\mathbb{Z}))$,
  we have
  \[
    \delta \langle a,b \rangle = \langle \delta a, b \rangle
    + (-1)^p \langle a, \delta b \rangle.
  \]
\end{proposition}
Since the proof is the straight forward calculation, we omit it.

\begin{proposition}\label{prop:cohomologous}
  The group cochains $c_k$ are cocycles and
  cohomologous to each other.
\end{proposition}

\begin{proof}
  By Proposition \ref{prop:chain_rule}, we have
  \begin{align*}
    \delta c_k &= \langle \delta \delta w_k, \Delta_k \rangle
    + (-1)^{n-k+1} \langle \delta w_k, \delta \Delta_k \rangle
    = 0
  \end{align*}
  for any $0 \leq k \leq n$.
  Thus the group cochain $c_k \in C_{\grp}^{n+1}(G;\mathbb{Z})$
  is a cocycle for any $0 \leq k \leq n$.
  For $0< k \leq n$, we have
  \begin{align*}
    c_k &= \langle \delta w_k, \Delta_k \rangle
    = -(-1)^{n-k+1} \langle dw_{k-1} , \Delta_k \rangle \\
    &= -(-1)^{n-k+1} \langle w_{k-1} , \partial \Delta_k \rangle
    = -(-1)^{n-k+1} \langle w_{k-1} , \delta \Delta_{k-1} \rangle \\
    &= - (\delta \langle w_{k-1}, \Delta_{k-1} \rangle
    - \langle \delta w_{k-1}, \Delta_{k-1} \rangle)
    = c_{k-1} - \delta \langle w_{k-1}, \Delta_{k-1} \rangle.
  \end{align*}
  Thus the cocycles $c_k$ are cohomologous to each other.
\end{proof}

\begin{remark}
  If $n = 2$ and the $2$-form $\Omega$ is an integral
  symplectic form on $M$,
  then the cocycle $c_2$ is, up to sign, equal to the cocycle
  introduced in \cite{maruyama19_DD}.
\end{remark}

Let $E_r^{p,q}$ denote the spectral sequence of the
double complex $C_{\grp}^p(G;C^q(M;\mathbb{Z}))$.
Then $E_2^{p,q}$ is isomorphic to $H_{\grp}^p(G;H^q(M;\mathbb{Z}))$,
where we consider the coefficients $H^q(M;\mathbb{Z})$
as the right $G$-module by pullback.
Since $H^q(M;\mathbb{Z}) = 0$ for $0 < q < n$,
we have
\[
  E_{n+1}^{0,n} = E_{n}^{0,n} = \cdots = E_2^{0,n}
  = H^n(M;\mathbb{Z})^G,
\]
where $H^n(M;\mathbb{Z})^G$ denotes the $G$-invariant part,
and
\[
  E_{n+1}^{n+1,0} = E_{n}^{n+1,0} = \cdots = E_2^{n+1,0}
  = H_{\grp}^{n+1}(G;\mathbb{Z}).
\]
Thus the transgression map
$d_{n+1}^{0,n} : E_{n+1}^{0,n} \to E_{n+1}^{n+1,0}$
defines the map
\[
  d_{n+1}^{0,n} : H^n(M;\mathbb{Z})^G \to
  H_{\grp}^{n+1}(G;\mathbb{Z}).
\]
Since the cohomology class $[w_n]$ is in $H^n(M;\mathbb{Z})^G$,
we obtain the class $d_{n+1}^{0,n} [w_n] \in H_{\grp}^{n+1}(G;\mathbb{Z})$.

\begin{proposition}\label{prop:cocycle-s.s.}
  The cohomology class $d_{n+1}^{0,n} [w_n]$ is equal to
  the class $[c_k]$.
\end{proposition}

\begin{proof}
  The transgression map $d_{n+1}^{0,n}$ is given by the coboundary of
  the tail of the zig-zag
  (see, for example, \cite[Section 14]{bott_tu82}).
  Moreover, the coboundary of the tail of zig-zag is equal to
  $\delta w_0 = \langle \delta w_0, \Delta_0 \rangle = c_0$.
  Thus we have $d_{n+1}^{0,n} [w_n] = [c_0]$.
  By Proposition \ref{prop:cohomologous}, we have
  $d_{n+1}^{0,n} [w_n] = [c_k] \in H_{\grp}^{n+1}(G;\mathbb{Z})$.
\end{proof}

\begin{corollary}\label{cor:independent}
  The cohomology class $[c_k]$ is independent of the choice of
  $w_k$ and $\Delta_k$.
\end{corollary}

Note that the above corollary can be
shown by a straight forward calculation.

\begin{remark}\label{remark:bounded}
  If we replace the coefficients $\mathbb{Z}$ with $\mathbb{R}$,
  then the class $d_{n+1}^{0,n}[\Omega]$
  in $H_{\grp}^{n+1}(G;\mathbb{R})$ is
  trivial since the zig-zag is trivial.
  Thus the cohomology class $d_{n+1}^{0,n}[w_n]$ is equal to $0$ in
  $H_{\grp}^{n+1}(G;\mathbb{R})$.
  By the exact sequence
  \[
    \cdots \to H_{\grp}^n(G;S^1) \to H_{\grp}^{n+1}(G;\mathbb{Z})
    \to H_{\grp}^{n+1}(G;\mathbb{R}) \to \cdots ,
  \]
  we have the class in $H_{\grp}^n(G;S^1)$ that hits to the class
  $d_{n+1}^{0,n}[w_n] = [c_k]$.
  Since the connecting homomorphism
  $H_{\grp}^n(G;S^1) \to H_{\grp}^{n+1}(G;\mathbb{Z})$
  factors through the bounded cohomology
  $H_b^{n+1}(G;\mathbb{Z})$,
  the cohomology class $[c_k]$ is bounded.
\end{remark}

\subsection{The group cocycles with coefficients in $S^1$}

By Remark \ref{remark:bounded}, we know the existence of the
cohomology class in $H_{\grp}^n(G;S^1)$ corresponding to the class
$[c_k] \in H_{\grp}^{n+1}(G;\mathbb{R})$
under the connecting homomorphism
\[
  \delta : H_{\grp}^n(G;S^1) \to H_{\grp}^{n+1}(G;\mathbb{Z}).
\]
In this section, we give cocycles
$b_k \in C_{\grp}^n(G;S^1)$
such that $\delta [b_k] = [c_k] \in H_{\grp}^{n+1}(G;\mathbb{Z})$.

By the assumption of the $n$-form $\Omega$ and the exact
sequence
\[
  \to H^n(M;\mathbb{Z}) \to H^n(M;\mathbb{R}) \xrightarrow{j}
  H^n(M;S^1) \to ,
\]
we have $j[\Omega] = 0$.
Here we consider $\Omega$ as the corresponding singular $n$-cocycle
(if we temporally use the symbol $\Omega_{\sing}$ to denote the
corresponding singular cocycle, this cocycle is
defined by $\Omega_{\sing} (\sigma) = \int_{\sigma} \Omega$
for any singular $n$-simplex $\sigma$).
We take a singular $(n-1)$-cochain $\eta_{n-1} \in C^{n-1}(M;S^1)$
such that $d\eta_{n-1} = j\Omega \in C^n(M;S^1)$.
By the universal coefficients theorem, the cohomology
$H^k(M;S^1)$ is trivial for $0 < k < n$.
Thus, as with the definition of
$w_k \in C_{\grp}^{n-k}(G;C^k(M;\mathbb{Z}))$,
we define group cochains
$\eta_k \in C_{\grp}^{n-1-k}(G;C^k(M;S^1))$ by
\[
  \delta \eta_k = -(-1)^{n-k} d\eta_{k-1}
\]
for $0 < k < n$.
Let $\Delta_k \in C_{\grp}^k(G;C_k(M;\mathbb{Z}))$ be the
cochains defined in Section \ref{subsec:cocycle}.
\begin{definition}\label{def:b_k}
  For $0 \leq k \leq n-1$, define
  $b_k \in C_{\grp}^n(G;S^1)$ by
  $b_k = \langle \delta \eta_k , \Delta_k \rangle$.
\end{definition}

As with Proposition \ref{prop:cohomologous} and Corollary
\ref{cor:independent}, we have
the following.
\begin{proposition}\label{prop:b_cohomologous}
  The group cochains $b_k$ are cocycles and
  cohomologous to each other.
  Moreover, the cohomology class $[b_k]$ is independent of the
  choice of cochains $\eta_k$ and $\Delta_k$.
\end{proposition}

\begin{theorem}
  Let $\delta : H_{\grp}^n(G;S^1) \to H_{\grp}^{n+1}(G;\mathbb{Z})$
  denote the connecting homomorphism, then we have
  $\delta [b_k] = [c_k]$.
\end{theorem}

\begin{proof}
  By Proposition \ref{prop:cohomologous} and \ref{prop:b_cohomologous},
  it is enough to show the equality $\delta [b_{n-1}] = [c_n]$.
  Recall that $b_{n-1} = \langle \delta\eta_{n-1},\Delta_{n-1} \rangle$ and
  $c_n = \langle \delta w_n,\Delta_n \rangle$.
  Let $\overline{\eta}_{n-1} \in C^{n-1}(M;\mathbb{R})$ be a lift of
  $\eta_{n-1} \in C^{n-1}(M;S^1)$, that is,
  $\overline{\eta}_{n-1}$ satisfies $j\overline{\eta}_{n-1}
  = \eta_{n-1}$,
  and put
  \[
    \overline{b}_{n-1} =
    \langle \delta \overline{\eta}_{n-1} , \Delta_{n-1} \rangle
    \in C_{\grp}^n(G;\mathbb{R}).
  \]
  Then we have
  \begin{align*}
    \delta \overline{b}_{n-1} &= \delta
    \langle \delta \overline{\eta}_{n-1} , \Delta_{n-1} \rangle
    =
    \langle \delta \delta \overline{\eta}_{n-1} , \Delta_{n-1} \rangle
    - \langle \delta \overline{\eta}_{n-1} , \delta\Delta_{n-1} \rangle\\
    & = - \langle \delta \overline{\eta}_{n-1} ,
    \partial\Delta_{n-1} \rangle
    = - \langle d\delta \overline{\eta}_{n-1} ,\Delta_{n-1} \rangle
    = - \langle \delta d\overline{\eta}_{n-1} ,\Delta_{n-1} \rangle.
  \end{align*}
  Since the action by $G$ preserves $\Omega$ as the singular $n$-cocycle,
  we have $\delta \Omega = 0 \in C_{\grp}^1(G;C^n(M;\mathbb{R}))$.
  Thus we have
  \begin{align*}
    \delta \overline{b}_{n-1} =
    - \langle \delta d\overline{\eta}_{n-1} ,\Delta_{n-1} \rangle
    = \langle \delta (\Omega - d\overline{\eta}_{n-1})
    ,\Delta_{n-1} \rangle.
  \end{align*}
  Since $j(\Omega - d\overline{\eta}_{n-1})
  = j\Omega - d\eta_{n-1} = 0$,
  the cocycle $\Omega - d\overline{\eta}_{n-1}$ is in
  $C^n(M;\mathbb{Z})$.
  This integer coefficients cocycle satisfies the assumption of $w_n$.
  Thus, if we put $w_n = \Omega - d\overline{\eta}_{n-1}
  \in C^n(M;\mathbb{Z})$, we have
  \[
    \delta \overline{b}_{n-1} = \langle \delta w_n , \Delta_n \rangle
    = c_n
  \]
  and this implies $\delta [b_{n-1}] = [c_n]$.
\end{proof}

\section{The Euler class of flat sphere bundles}

In this section, for the $n$-sphere $S^n$ and the normalized
standard volume form,
we show that the class $[c_k]$ is equal to the Euler class
of flat sphere bundles up to sign (Theorem \ref{thm:cocycle_euler_class})
and show that
the group cohomology classes
$[c_k]$ and $[b_k]$ are non-trivial
(Theorem \ref{thm:non-trivial_example}).

Let us recall that the construction of the Euler class in terms of
the Leray-Serre spectral sequence.
Let $E \to B$ be an oriented sphere bundle over a connected
base space $B$ and $E_r^{p,q}$
denote the Leray-Serre spectral sequence.
Since $H^k(S^n;\mathbb{Z}) = 0$ for $0<k<n$
and the bundle is oriented,
we have
$E_{n+1}^{0,n} = E_2^{0,n} = H^n(S^n;\mathbb{Z})$
and
$E_{n+1}^{n+1,0} = E_2^{n+1,0} = H^{n+1}(B;\mathbb{Z})$.
Let $d_{n+1}^{0,n} : E_{n+1}^{0,n}
\to E_{n+1}^{n+1,0}$ denote the derivation map
and $\theta$ denote
the generator of the cohomology $H^n(S^n;\mathbb{Z})
= E_{n+1}^{0,n}$,
then the cohomology class $-d_{n+1}^{0,n} \theta \in H^{n+1}(B;\mathbb{Z})$
is the Euler class of the oriented sphere bundle
$E \to B$.

Let $\Diff_{+}(S^n)$ denote the orientation-preserving
diffeomorphism group and
$e \in H^{n+1}(B\Diff_{+}(S^n)^{\delta};\mathbb{Z}) \cong
H_{\grp}^{n+1}(\Diff_{+}(S^n);\mathbb{Z})$
denote the universal Euler class of flat sphere bundles.
Let us consider the normalized standard volume form
$\Omega$ on $S^n$.
Then the volume-preserving diffeomorphism group
$G = \Diff_{\Omega}(S^n)$ is included in $\Diff_{+}(S^n)$.
Let $e_G \in H_{\grp}^{n+1}(G;\mathbb{Z})$ denote
the pullback of the Euler class.
By the naturality of the Euler class, the class $e_G$
is the universal Euler class of flat sphere bundles
whose structure group is reduced to $G$.

\begin{theorem}\label{thm:cocycle_euler_class}
  The cohomology class $[c_k] \in H_{\grp}^{n+1}(G;\mathbb{Z})$ is
  equal to the negative of the Euler class $e_G$.
\end{theorem}

\begin{proof}
  Let $EG^{\delta} \to BG^{\delta}$
  denote the universal $G^{\delta}$-bundle.
  Then the Borel construction
  $S_{G^{\delta}}^n = EG^{\delta} \times_{G^{\delta}} S^n \to BG^{\delta}$
  is the universal flat sphere bundle.
  Note that the Leray-Serre spectral sequence
  of the Borel construction is isomorphic to
  the spectral sequence used in Section \ref{subsec:cocycle}
  (see \cite{losik93}).
  Thus, by the construction of the Euler class
  in the Leray-Serre spectral sequence,
  the class $d_{n+1}^{0,n}[w_n] \in H_{\grp}^{n+1}(G;\mathbb{Z})$
  is equal to the negative of
  the Euler class of the flat sphere bundle
  $S_{G^{\delta}}^n \to BG^{\delta}$
  under the
  identification $H_{\grp}^{n+1}(G;\mathbb{Z}) \cong
  H^{n+1}(BG^{\delta};\mathbb{Z})$.
  By Proposition \ref{prop:cocycle-s.s.},
  the group cocycles $c_k$ give the negative of the
  Euler class of the universal flat sphere bundle.
\end{proof}

Since the natural action by $SO(n+1)$ on $S^n$ preserves the normalized
standard volume form $\Omega$, there is the inclusion
$SO(n+1) \to G$.
Let $e_{SO(n+1)} \in H_{\grp}^{n+1}(SO(n+1);\mathbb{Z})$ denote
the pullback of $e_G$ by the inclusion.
By the naturality of the Euler class, the class
$e_{SO(n+1)}$ is the universal Euler class of flat sphere bundles
whose structure group is reduced to $SO(n+1)$.
The universal Euler class in $H^{n+1}(BSO(n+1);\mathbb{Z})$
of vector bundles hits the class $e_{SO(n+1)}$ under
the canonical map
\[
  H^{n+1}(BSO(n+1);\mathbb{Z}) \to H^{n+1}(BSO(n+1)^{\delta};\mathbb{Z})
  \cong H_{\grp}^{n+1}(SO(n+1);\mathbb{Z}).
\]
Since the canonical map is injective (see \cite{milnor83}),
the class $e_{SO(n+1)}$ is non-trivial and
so is the class $e_G$.
Thus, we obtain the following theorem.
\begin{theorem}\label{thm:non-trivial_example}
  Let $M$ be the $n$-sphere and $\Omega$ the normalized
  standard volume form.
  Then the classes $[c_k] \in H_{\grp}^{n+1}(G;\mathbb{Z})$
  and $[b_k] \in H_{\grp}^n(G;S^1)$ are non-trivial.
\end{theorem}

\bibliographystyle{amsplain}
\bibliography{group_cocycle.bib}
\end{document}